\documentclass[12pt]{amsart}
\usepackage{latexsym, amssymb, amsmath}

\newtheorem{thm}{Theorem}[section]

\newtheorem{rem}[thm]{Remark}
\theoremstyle{definition}

\newcommand{\Ric}{{\rm Ric}}

\begin{document}

\title{Complete Yamabe solitons with finite total scalar curvature}



\author{Shun Maeta}
\address{Department of Mathematics,
 Shimane University, Nishikawatsu 1060 Matsue, 690-8504, Japan.}
\curraddr{}
\email{shun.maeta@gmail.com~{\em or}~maeta@riko.shimane-u.ac.jp}
\thanks{The author is partially supported by the Grant-in-Aid for Young Scientists(B), No.15K17542, Japan Society for the Promotion of Science, and JSPS Overseas Research Fellowships 2017-2019 No. 70.}


\date{}

\dedicatory{}


\commby{}

\begin{abstract}
In this paper, we show that steady or shrinking complete gradient Yamabe solitons with finite total scalar curvature and non-positive Ricci curvature are Ricci flat.
Moreover, under certain pinching condition for Ricci curvature, we show that steady or shrinking complete gradient Yamabe solitons with finite total scalar curvature and non-positive scalar curvature have zero scalar curvature.
\end{abstract}

\maketitle


\bibliographystyle{amsplain}

\section{Introduction}\label{intro} 
 A Riemannian manifold $(M^n,g)$ is called a gradient Yamabe soliton if there exists a smooth function $F$ on $M$ and a constant $\rho\in \mathbb{R}$, such that 
\begin{equation}\label{YS}
(R-\rho)g=\nabla\nabla F,
\end{equation}
where $R$ is the scalar curvature on $M$.
If $\rho>0$, $\rho=0$ or $\rho<0$, Yamabe solitons are called shrinking, steady or expanding (cf. \cite{DS13}).
By scaling the metric, we can assume $\rho=1,0,-1$, respectively. If the potential function $F$ is constant, then Yamabe solitons are called trivial.
Yamabe solitons are special solutions of the Yamabe flow which was introduced by R. Hamilton \cite{Hamilton89}.
Since it is well known that any compact Yamabe soliton is trivial (cf.~\cite{Hsu12}), 
it is important to study complete Yamabe solitons.
 P. Daskalopoulos and N. Sesum showed that all locally conformally flat complete gradient Yamabe solitons with positive sectional curvature have to be rotationally symmetric (cf.~\cite{DS13}).
 There are many interesting studies for Yamabe soliton and CR Yamabe solitons (see for example Pak Tung Ho's paper \cite{Ho1}, \cite{Ho2}.)
Under certain integrality conditions, L. Ma and V. Miquel gave some conditions for the scalar curvature (cf.~\cite{MM12}).
L. Ma and L. Cheng's paper \cite{MC11} is also important.
Recently, H.-D. Cao, X. Sun and Y. Zhang gave a classification theorem for complete steady gradient Yamabe solitons under the Ricci curvature is positive (cf. \cite{CSZ12}).
Therefore, in this paper, under the Ricci curvature is non-positive, 
we show the following.

\begin{thm}\label{main}
Let $(M,g)$ be a steady or shrinking complete gradient Yamabe soliton with non-positive Ricci curvature.
If the scalar curvature $R\in L^p(M)$ for some $0<p<\infty$,
then $M$ is Ricci flat.
\end{thm}

Under the scalar curvature is non-positive and certain pinching condition for Ricci curvature of $M$, we show the following.

\begin{thm}\label{main2}
Let $(M,g)$ be a steady or shrinking complete gradient Yamabe soliton with non-positive scalar curvature and $\Ric\geq \varphi Rg$ for some non-negative function $\varphi :M\rightarrow \mathbb{R}^+\cup\{0\}$.
If the scalar curvature $R\in L^p(M)$ for some $0<p<\infty$,
then the scalar curvature is $0$.
\end{thm}

\begin{rem}
If the trace free Ricci curvature $\overline{Ric}=\Ric-\frac{1}{n}Rg$ is non-negative, then it satisfies the assumption $\Ric\geq \varphi Rg$ for $\varphi=\frac{1}{n}$.
\end{rem}
\section{Proof of Theorem $\ref{main}$ and $\ref{main2}$}
In this section, we give a proof of Theorem $\ref{main}$ and Theorem $\ref{main2}$.
\begin{proof}[Proof of Theorem $\ref{main}$]
In general, we have
\begin{equation}\label{p.1}
\Delta {\nabla}_iF={\nabla}_i\Delta F+\Ric_{ij}{\nabla}_jF,
\end{equation}
where $\nabla$ is the Levi-Civita connection, $\Delta$ is the Laplacian, and $\Ric$ is the Ricci tensor on $M$, respectively.
Substituting 
\begin{equation*}
\Delta {\nabla}_iF={\nabla}_k{\nabla}_k{\nabla}_iF={\nabla}_k((R-\rho)g_{ki})={\nabla}_iR,
\end{equation*}
and
\begin{equation*}
{\nabla}_i\Delta F={\nabla}_i(n(R-\rho))=n{\nabla}_iR,
\end{equation*}
into $(\ref{p.1})$, we have
\begin{equation}\label{p.2}
(n-1)\nabla_iR+\Ric_{il}\nabla_lF=0.
\end{equation}
Thus we have 
\begin{equation}\label{p.3}
(n-1)g(\nabla R,\nabla F)=-\Ric(\nabla F,\nabla F).
\end{equation}
On the other hand, by $(\ref{p.2})$ and the contracted second Bianchi identity, 
\begin{equation}\label{p.4}
(n-1)\Delta R+\frac{1}{2}g(\nabla R, \nabla F)+R(R-\rho)=0.
\end{equation}
Combining $(\ref{p.3})$ and $(\ref{p.4})$, we obtain
\begin{equation}
\Delta R=\frac{1}{2(n-1)^2}\Ric(\nabla F,\nabla F)-\frac{1}{n-1}R(R-\rho).
\end{equation}
Thus we have 
\begin{align}\label{KEY}
\Delta R^2
=&~2R\Delta R +2|\nabla R|^2\\\notag
=&~ \frac{1}{(n-1)^2}R\,\Ric(\nabla F,\nabla F)-\frac{2}{n-1}R^2(R-\rho)\\\smallskip\notag
&~+2|\nabla R|^2.\notag
\end{align}

Set $f=R^2$. 
For a fixed point $p_0\in M$, and for every $0<r<\infty,$
 we first take a cut off function $\eta$ on $M$ satisfying that 
 \begin{equation}
\left\{
 \begin{aligned}
&0\leq\eta(p)\leq1\ \ \ (p\in M),\\
&\eta(p)=1\ \ \ \ \ \ \ \ \ (p\in B_r(p_0)),\\
&\eta(p)=0\ \ \ \ \ \ \ \ \ (p\not\in B_{2r}(p_0)),\\
&|\nabla\eta|\leq\frac{C}{r}\ \ \ \ \ \ \ (p\in M),\ \ \ \text{for some constant $C$ independent of $r$},
\end{aligned} 
\right.
\end{equation}
where $B_r(p_0)$  are the balls centered at a fixed point $p_0\in M$ with radius $r$.
Let $a$ be a positive constant to be determined later.
 Let $b=\frac{(a+\frac{3}{2})(a+3-d)}{a+\frac{3}{2}-d}$, where $d<a+1$ is a positive constant. 

By $(\ref{KEY})$, we have
\begin{equation}\label{1}
\begin{aligned}
-\int_M & g( \nabla (\eta ^b  f^a), \nabla f) dv_g\\
=&\int_M \eta^b f^a \Delta f dv_g\\
=& \int_M \frac{1}{(n-1)^2}\eta^b f^a \,R\,\Ric(\nabla F,\nabla F) dv_g 
-\int_M \frac{2}{n-1}\eta^b f^{a+1}(R-\rho)dv_g \\
&+2\int_M \eta^b f^a|\nabla R|^2dv_g.
\end{aligned}
\end{equation}
On the other hand,
\begin{equation}\label{2}
\begin{aligned}
-\int_M & g(\nabla (\eta ^b  f^a), \nabla f) dv_g\\
&= -2 b \sum_{i=1}^m\int_M \eta^{b-1} (e_i \eta) f^a \,R\,(e_i R) dv_g 
-4a \int_M \eta^{b} f^{a} |\nabla R|^2dv_g.
\end{aligned}
\end{equation}
From $(\ref{1})$ and $(\ref{2})$, we obtain

\begin{equation}\label{3}
\begin{aligned}
& \int_M \frac{1}{(n-1)^2}\eta^b f^a \,R\,\Ric(\nabla F,\nabla F) dv_g 
-\int_M \frac{2}{n-1}\eta^b f^{a+1}(R-\rho)dv_g \\
&+(2+4a)\int_M \eta^b f^a|\nabla R|^2dv_g\\
=&-2 b \sum_{i=1}^m\int_M \eta^{b-1} (e_i \eta) f^a \,R\,(e_i R) dv_g \\
=&-2  \sum_{i=1}^m\int_M (\eta^{\frac{b}{2}}f^{\frac{a}{2}} (e_i R))(b \eta^{\frac{b-2}{2}}f^\frac{a}{2} \,R\,(e_i \eta) dv_g \\
\leq& \int_M \eta^b f^a |\nabla R|^2dv_g+\int_M b^2 \eta^{b-2}f^{a+1} |\nabla \eta|^2 dv_g,
\end{aligned}
\end{equation}
where the last inequality follows from Young's inequality.
By using Young's inequality again, we have

\begin{equation}\label{4}
\begin{aligned}
&\int_M b^2 \eta^{b-2}f^{a+1} |\nabla \eta|^2 dv_g\\
 =& \int_M \eta^c b^2 \eta^{b-2-c} f^{d} f^{a+1-d} |\nabla \eta|^2 dv_g \\
 \leq & \frac{2}{n-1}\int_M \eta ^{b} f^{a+\frac{3}{2}} dv_g\\
& \ \ \ +C(a,d,n) \int_M \eta ^{(b-2-c)\frac{a+\frac{3}{2}}{a+\frac{3}{2}-d}} f^{(a+1-d)\frac{a+\frac{3}{2}}{a+\frac{3}{2}-d}} |\nabla\eta|^{2\frac{a+\frac{3}{2}}{a+\frac{3}{2}-d}} dv_g\\
 = & \frac{2}{n-1}\int_M \eta ^{b} f^{a+1}f^{\frac{1}{2}} ~dv_g\\
& \ \ \ +C(a,d,n) \int_M \eta ^{a+1-d} f^{(a+1-d)\frac{a+\frac{3}{2}}{a+\frac{3}{2}-d}} |\nabla\eta|^{2\frac{a+\frac{3}{2}}{a+\frac{3}{2}-d}} dv_g,
\end{aligned}
\end{equation}
where $c=\frac{d(a+3-d)}{a+\frac{3}{2}-d}$ and $C(a,d,n)$ is a constant depending only on $a$, $d$ and $n$.
Combining $(\ref{3})$ and $(\ref{4})$, we obtain
\begin{equation}\label{5}
\begin{aligned}
& \int_M \frac{1}{(n-1)^2}\eta^b f^a \,R\,\Ric(\nabla F,\nabla F) dv_g 
-\int_M \frac{2}{n-1}\eta^b f^{a+1}R\,dv_g \\
&+\int_M \frac{2\rho}{n-1}\eta^b f^{a+1}dv_g +(1+4a)\int_M \eta^b f^a|\nabla R|^2dv_g\\
\leq& \frac{2}{n-1}\int_M \eta ^{b} f^{a+1}f^{\frac{1}{2}} ~dv_g+C(a,d,n) \int_M \eta ^{a+1-d} f^{(a+1-d)\frac{a+\frac{3}{2}}{a+\frac{3}{2}-d}} |\nabla\eta|^{2\frac{a+\frac{3}{2}}{a+\frac{3}{2}-d}} dv_g\\
 \leq &
\frac{2}{n-1}\int_M \eta ^{b} f^{a+1}f^{\frac{1}{2}} ~dv_g
+C(a,d,n) \int_M f^{(a+1-d)\frac{a+\frac{3}{2}}{a+\frac{3}{2}-d}} \left(\frac{1}{r}\right)^{2\frac{a+\frac{3}{2}}{a+\frac{3}{2}-d}} dv_g\\
= &
\frac{2}{n-1}\int_M \eta ^{b} f^{a+1}f^{\frac{1}{2}} ~dv_g
+C(a,d,n) \int_M f^{\frac{p}{2}} \left(\frac{1}{r}\right)^{2\frac{a+\frac{3}{2}}{a+\frac{3}{2}-d}} dv_g,
\end{aligned}
\end{equation}
 where we chose $a$ and $d$ such that $p=2{(a+1-d)\frac{a+\frac{3}{2}}{a+\frac{3}{2}-d}}$.
 Since $R\leq0$, we have
 \begin{equation}\label{6}
\begin{aligned}
& \int_M \frac{1}{(n-1)^2}\eta^b f^a \,R\,\Ric(\nabla F,\nabla F) dv_g \\
&+\int_M \frac{2\rho}{n-1}\eta^b f^{a+1}dv_g +(1+4a)\int_M \eta^b f^a|\nabla R|^2dv_g\\
\leq &~
C(a,d,n) \int_M f^{\frac{p}{2}}  \left(\frac{1}{r}\right)^{2\frac{a+\frac{3}{2}}{a+\frac{3}{2}-d}} dv_g.
\end{aligned}
\end{equation}
  Since $0<d<a+1$, $0<p=2(a+1-d)\frac{a+2}{a+2-d}<2(a+1)$.
By the assumption $\int_M f^{\frac{p}{2}} dv_g<\infty$ $(0<p<\infty)$, letting $r\nearrow \infty$ in $(\ref{6})$,
  the right hand side of $(\ref{6})$ goes to zero and the left hand side of $(\ref{6})$ goes to 
\begin{align*}
& \int_M \frac{1}{(n-1)^2} f^a \,R\,\Ric(\nabla F,\nabla F) dv_g \\
&+\int_M \frac{2\rho}{n-1} f^{a+1}dv_g +(1+4a)\int_M f^a|\nabla R|^2dv_g,
\end{align*}
   since $\eta=1$ on $B_r(p_0)$.
 Thus, we have
\begin{equation}\label{7}
\begin{aligned}
& \int_M \frac{1}{(n-1)^2} f^a \,R\,\Ric(\nabla F,\nabla F) dv_g \\
&+\int_M \frac{2\rho}{n-1} f^{a+1}dv_g +(1+4a)\int_M f^a|\nabla R|^2dv_g\leq0.
\end{aligned}
\end{equation}
 Thus, we obtain $R$ is constant. By $(\ref{p.4})$, $R=0$ or $R=\rho.$

 Case~$1$: $M$ is a shrinking Yamabe soliton $(\rho=1)$. By $(\ref{7})$, we have $R=0$.
 
 Case~$2$: $M$ is a steady Yamabe soliton $(\rho=0)$. Obviously $R=0.$
 
  Therefore, we obtain $R=0.$ From this and the assumption that the Ricci curvature is non-positive, we obtain the Ricci curvature is $0$.
\end{proof}

\quad\\
By the similar argument as in the proof of Theorem $\ref{main}$, we show Theorem $\ref{main2}$.

\begin{proof}[Proof of Theorem $\ref{main2}$]
By the same argument as in the proof of Theorem $\ref{main},$ we have
 \begin{equation}\label{r.6}
\begin{aligned}
& \int_M \frac{1}{(n-1)^2}\eta^b f^a \,R\,\Ric(\nabla F,\nabla F) dv_g \\
&+\int_M \frac{2\rho}{n-1}\eta^b f^{a+1}dv_g +(1+4a)\int_M \eta^b f^a|\nabla R|^2dv_g\\
\leq &~
C(a,d,n) \int_M f^{\frac{p}{2}}  \left(\frac{1}{r}\right)^{2\frac{a+\frac{3}{2}}{a+\frac{3}{2}-d}} dv_g.
\end{aligned}
\end{equation}
Since the assumption $\Ric\geq \varphi Rg$, we have
 \begin{equation}\label{r.7}
\begin{aligned}
& \int_M \frac{1}{(n-1)^2}\eta^b \varphi f^{a+1}|\nabla F|^2 dv_g \\
&+\int_M \frac{2\rho}{n-1}\eta^b f^{a+1}dv_g +(1+4a)\int_M \eta^b f^a|\nabla R|^2dv_g\\
\leq &~
C(a,d,n) \int_M f^{\frac{p}{2}}  \left(\frac{1}{r}\right)^{2\frac{a+\frac{3}{2}}{a+\frac{3}{2}-d}} dv_g.
\end{aligned}
\end{equation}
By the assumption $\int_M f^{\frac{p}{2}} dv_g<\infty$ $(0<p<\infty)$, letting $r\nearrow \infty$ in $(\ref{r.7})$,
  the right hand side of $(\ref{r.7})$ goes to zero and the left hand side of $(\ref{r.7})$ goes to 
\begin{align*}
& \int_M \frac{1}{(n-1)^2} \varphi f^{a+1}|\nabla F|^2 dv_g \\
&+\int_M \frac{2\rho}{n-1} f^{a+1}dv_g +(1+4a)\int_M  f^a|\nabla R|^2dv_g,
\end{align*}
   since $\eta=1$ on $B_r(p_0)$.
 Thus, we have
\begin{equation}\label{r.8}
\begin{aligned}
&\int_M \frac{1}{(n-1)^2} \varphi f^{a+1}|\nabla F|^2 dv_g \\
&+\int_M \frac{2\rho}{n-1} f^{a+1}dv_g +(1+4a)\int_M  f^a|\nabla R|^2dv_g\leq 0.\
\end{aligned}
\end{equation}
 Therefore, we obtain $R$ is constant. By $(\ref{p.4})$, $R=0$ or $R=\rho.$

 Case~$1$: $M$ is a shrinking Yamabe soliton $(\rho=1)$. By $(\ref{r.8})$, we have $R=0$.
 
 Case~$2$: $M$ is a steady Yamabe soliton $(\rho=0)$. Obviously $R=0.$
 
\end{proof}

\proof[Acknowledgements]
The author would like to thank Professor Pak Tung Ho and Professor Li Ma for useful comments.

\bibliographystyle{amsbook}

\begin{thebibliography}{99}  

\bibitem{CSZ12}
Cao, H.-D., Sun, X. and Zhang, Y.,
{\em On the structure of gradient Yamabe solitons,}
Math. Res. Lett. (2012) {\bf 19}, 767--774.

\bibitem{CLN06}
Chow, B., Lu P. and  Ni L.,
{\em Hamilton's Ricci Flow},
Graduate Studies in Mathematics {\bf 77},  Amer. Math. Soc., (2006). 

\bibitem{DS13}
Daskalopoulos, P. and Sesum, N.,
\textit{The classification of locally conformally flat Yamabe solitons,}
Adv. Math. 240, (2013) 346--369.

\bibitem{Hamilton89}
Hamilton, R., 
{\em Lectures on geometric flows},
(1989), unpublished.

\bibitem{Ho1}
Ho, P. T.,
{\em A note on compact CR Yamabe solitons},
J. Geom. Phys. 94, (2015) 32-34.

\bibitem{Ho2}
Ho, P. T.,
{\em Soliton to the fractional Yamabe flow},
Nonlinear Anal. 139, (2016)  211-217.

\bibitem{Hsu12}
Hsu, S. Y.,
\textit{A note on compact gradient Yamabe solitons,}
J. Math. Anal. Appl. 388 (2), (2012) 725--726.

\bibitem{MC11}
Ma, L. and Cheng, L.,
{\em Properties of complete non-compact Yamabe solitons},
 Ann. Global Anal. Geom. 40, (2011) 379-387 .

\bibitem{MM12}
Ma, L. and Miquel, V.,
\textit{Remarks on scalar curvature of Yamabe solitons,}
Ann. Global Anal. Geom. 42,  (2012) 195-205.


\end{thebibliography}

\end{document}